\documentclass{amsart}
\usepackage{amsmath, amssymb, epsfig, graphicx, color }

\numberwithin{equation}{section}

\newtheorem{thm}{Theorem}[section]

\newtheorem{lem}[thm]{Lemma}

\newtheorem{prob}[thm]{Problem}

\newcommand{\be}{\begin{equation}}
\newcommand{\ee}{\end{equation}}

\newcommand{\bochner}{{\mathcal B}}
\newcommand{\pettis}{{\mathcal P}}
\newcommand{\nat}{{\mathbb N}}
\newcommand{\reals}{{\mathbb R}}

\newcommand{\2}{\{0,1\}}
\newcommand{\infseq}{\2 ^{\infty}}
\newcommand{\finseq}{\2 ^{< \infty}}
\newcommand{\n}{\bold n}
\newcommand{\fn}{f(\n)}

\newcommand{\B}{{\bold B}}
\newcommand{\D}{{\bold D}}
\newcommand{\nd}{\bold {ND}}

\begin{document}

\title{Lineability of non-differentiable Pettis primitives}
\author{  B. Bongiorno}
\address{Department of Mathematics, University of Palermo, Via Archirafi 34, 90123 Palermo, Italy}
\email{bbongi@math.unipa.it}
\author{U. B. darji}
\address{Department of Mathematics, University of Louisville,  Louisville,  KY 40292, USA}
\email{ubdarj01@louisville.edu}
\author{L. Di Piazza}
\address{Department of Mathematics, University of Palermo, Via Archirafi 34, 90123 Palermo, Italy}
\email{dipiazza@math.unipa.it}
\date{}
\thanks{The second author would like to thank the hospitality of the department of Mathematics of University of Palermo and grant Cori 2013 of the University of Palermo}
\maketitle

\begin{abstract}Let $X$ be an infinite-dimensional Banach space. In 1995, settling a long outstanding problem of Pettis, Dilworth and Girardi constructed an $X$-valued Pettis integrable function on $[0,1]$ whose primitive is nowhere weakly differentiable. Using their technique and some new ideas we show that $\nd$, the set of strongly measurable Pettis integrable functions with nowhere weakly differentiable primitives, is lineable, i.e., there is an infinite dimensional vector space whose nonzero vectors belong to $\nd$.
\end{abstract}
\section{Introduction}
Throughout this note $X$ is an infinite dimensional Banach space. For $X$-valued functions there are essentially two distinct notions of integration: the Bochner  integral and the Pettis integral. The latter one includes properly the Bochner integral and preserves some of its good properties, e.g, countable additivity and absolute continuity of the indefinite integral, and convergence theorems.  For $X$-valued functions on $[0,1]$ the Bochner indefinite integral is almost everywhere differentiable.  As Pettis himself pointed out  \cite{pettis},  the same property is not enjoyed by the Pettis integral.  An  interesting  problem left open  in  \cite{pettis} was whether the indefinite Pettis integral of a strongly measurable Pettis  integrable function $f$ is
 almost everywhere weakly differentiable (that is, does there exist a set  $E \subset[0,1]$ of full measure such that  $\int_0^tx^*f$ is differentiable to $x^*f$ for each $t\in E$ and for each $x^*$ in the dual of $X$).

 After a string of partial results beginning 1940 (\cite{ph}, \cite{munroe}, \cite {dg1}, \cite{kadets}),   this problem was settled decisively and beautifully in 1995 by Dilworth and Girardi \cite{girardi}. They exhibited that every infinite dimensional Banach-space admits a Pettis integrable function from $[0,1]$ into $X$ whose primitive is nowhere weakly differentiable. Their proof is rather flexible and gives the impression that there are many such functions. How does one make such a statement precise?  One possibility is to show that in the space of all strongly measurable function from $[0,1]$ into $X$, the set of Pettis integrable functions with nowhere weakly differentiable primitive is a dense $G_{\delta}$ set. This was done by Popa  \cite{popa} in 2000. The topology one uses in this setting is the topology generated by the Pettis norm. The shortcoming of this method is that the Pettis norm is not complete \cite{dg1}. Hence, proving that the set of Pettis integrable functions with nowhere weakly differentiable primitive is a dense $G_{\delta}$ set loses some of its significance.

 An alternate notion of bigness in Banach space was introduced by Gurarity \cite{g} and followed up in \cite{a1} \cite{a2} \cite{gq} . This notion of bigness is of algebraic nature. If $X$ is a Banach space then a subset $M$ of $X$ is {\it lineable} if $M\cup \{0\}$ contains an infinite dimensional vector space. If, moreover, this infinite dimensional vector space is closed in the norm topology, then $M$ is said to be {\it spaceable}. During the last twenty years many classical, pathological subsets of Banach spaces have been shown to be lineable or spaceable. What is surprising is that most of these sets are far from being vector spaces. For example,  Gurarity \cite{g2}
 showed that the space of continuous nowhere differentiable functions on $[0,1]$ is lineable. The spaceability of this set  was shown in \cite{fgk}. Roder\'iguez-Piazza  \cite{r} showed that every separable Banach space is isometric to a subspace of $C[0,1]$ whose every non-zero element is nowhere differentiable. Later this result was strengthened by Hencl  \cite{h} who showed that one can replace ``nowhere differentiable" by ``nowhere approximately differentiable and nowhere H\"older function".  More recently, the spaceability of nowhere integrable functions was shown by Glab, Kaufman and Pellegrini \cite{gkp}. For a survey of results concerning lineability and spaceability as well as many interesting results and open problems, we refer the reader to the recent paper by Enflo et al. \cite{enflo}.

In this paper we study, from the viewpoint of lineability, the set of strongly measurable Pettis integrable functions whose Pettis integral is nowhere weakly differentiable. We show (Theorem \ref{mainthm})  that for  every infinite dimensional Banach-space $X$ there  is a linearly independent set  ${\mathcal V}$ of  strongly measurable Pettis integrable $X$-valued functions, satisfying the property  that ${\mathcal V}$ has  the cardinality of the continuum and for each sequence $\{\lambda_k\} \in \ell_1$ and each  sequence of $\{f_k\}$ in ${\mathcal V}$, the function
$ f= \sum_{k=1}^{\infty} \lambda_k f_k$ is Pettis integrable and its primitive is nowhere weakly differentiable, provided that it is not the zero function. Hence, the lineability of $\nd$ follows. Our techniques  use ideas of  Dilworth and Girardi as well as some interesting applications of the Dvoretzky's theorem  and some basic tools of infinite combinatorics. The following problems remain open.

\begin{prob} Is $\nd$ spaceable?
\end{prob}
If the answer to the above questions is yes, can more be shown?
\begin{prob} Let $X$ be an infinite dimensional Banach space. Can every separable Banach space be isometrically embedded into $\nd \cup \{0\}$?
\end{prob}

The paper is organized as follows: Section 2 contains definitions and some basic facts concerning vector-valued integration. The main result of this article, Theorem \ref{mainthm}, is also stated there. The rest of the paper concerns the proof of this theorem.  We have decided to give proof of the main results in two parts. In Section~3 we give a proof in the case when $X= \ell_2$.  The reason for giving a separate proof for $\ell_2$ is that this case does not use Dvoretzky's Theorem. Hence, the combinatorics and the estimates are easier to follow. Building up on the notation, ideas and  techniques of Section~3, in Section~4 we give the proof  of the general case.

\section{Basic facts}

Throughout $X$ is a  infinite dimensional Banach space and $X^*$ is its dual space. $[0,1]$ is the unit interval of the real line equipped with  the Lebesgue measure $\lambda$.   By $\mathcal{L}$ we denote  the family of all Lebesgue measurable subsets of $[0,1]$.

We recall that a strongly measurable function $f:[0,1] \rightarrow X$ is said to be {\it Bochner integrable} in $[0,1]$ if $\int_0^1 ||f|| < \infty$. A   function $f:[0,1] \rightarrow X$ is said to be {\it  Pettis integrable} in $[0,1]$ if $x^*f \in L_1$ for all $x^* \in X^*$ and for each $E \in \mathcal{L}$ there exists a vector $x_E \in X$ such that $x^*(x_E) = \int_E x^*f $, for  all $x^* \in X^*$. We write $x_E = (P)\int_E f $ and we call it the {\it Pettis integral of } $f$ over $E$. Moreover we call   {\it Pettis primitive  } or simple   {\it  primitive  } of $f$ the function $F(t) = (P)\int_0^t f$. We refer the reader to \cite{du} for basic theory of vector-valued integration.

 We use $\pettis$ (resp.  $\bochner$)  to denote the set of all strongly measurable Pettis   integrable  (resp. Bochner integrable)  functions from $[0,1]$ into $X$. Recall that $\bochner \subsetneq \pettis$.

We have the following basic fact.
\begin{lem}{\rm (see \cite[Corollary 5.1]{mu})}\label{basic1} Let $f:[0,1] \rightarrow X$ be defined by
\[ f = \sum_{k=1}^{\infty} x_k \chi_{A_k}\]
where $\{A_k\}$  is a sequence of pairwise disjoint sets of  $\mathcal{L}$  and $\{x_n\}$ a sequence in $X$.
\begin{enumerate}
\item $f \in \bochner $ iff $\sum_{k=1}^{\infty} \| x_k\| \lambda(A_k) < \infty$.
\item $f  \in \pettis $ is Pettis integrable iff $\sum_{k=1}^{\infty} x_k \lambda(A_k) $ is unconditionally
convergent in $X$.
\end{enumerate}
Moreover, if either of the integrals exists, then it equals $\sum_{k=1}^{\infty} x_k \lambda(A_k) $.
\end{lem}

The following is the  main theorem of this article.

\begin{thm}\label{mainthm} There exists a set ${\mathcal V} \subset \pettis$ such that
\begin{enumerate}
\item ${\mathcal V}$ is the size of the cardinality of the continuum,
\item ${\mathcal V}$ is linearly independent,
\item for each sequence $\{\lambda_k\} \in \ell_1$ and a sequence of $\{f_k\}$ in ${\mathcal V}$
we have that
\[ f= \sum_{k=1}^{\infty} \lambda_k f_k
\] is in $\pettis$, and
\item  moreover, if $f$ is not the zero function, then the primitive  $F$ of $f$,  has the property that for all $x \in [0,1]$, we have
\[ \limsup_{h\rightarrow 0} \left \| \frac{F(x+h)-F(x)}{h} \right \|= \infty
\]
\end{enumerate}
\end{thm}

\vspace{2ex}

\section{The case $X =\ell_2 $}
The proof  in the case of  $X =\ell_2 $ is simpler and does not make a  use Dvoretzky's theorem.
We present this proof first to give the general idea. We will then prove the general case.\\

 We now introduce some terminology and notation
necessary for the proof.\\

By  $\infseq$,  $\finseq$ we denote the set of all infinite and
the set of all finite (including the empty) sequences of $\2$, respectively.  For each $\sigma \in
\finseq$, we let $|\sigma|$ denote the length of $\sigma$ and $\sigma i$, $i \in \2$, denote the extension of $\sigma$ by $i$.
If $\sigma \in \finseq \cup \infseq$, and $|\sigma| \ge i$, then $\sigma_{|i}$ denotes the restriction of $\sigma$ to
the first $i$ terms. If $\tau \in \finseq$, then $[\tau]$ denotes the set of $\sigma \in \finseq$ which are extension of
$\tau$, namely the set of all $\sigma \in \finseq$ such that $\sigma_{|i} =\tau$ where $i = | \tau|$.
We let
\[ \B =\{(\sigma, i)| \sigma \in \finseq, 0 \le i \le |\sigma| \}.
\]

For each $\sigma \in \finseq$, we define a closed interval $I_{\sigma} \subseteq [0,1]$ of length $2^{-|\sigma|}$
recursively in the following fashion: If $\sigma$ is the empty sequence, then $I_{\sigma} = [0,1]$. In general, if $I_{\sigma}$ is defined then
$I_{\sigma 0}$ is the left half of $I_{\sigma}$ and $I_{\sigma 1}$ is the right half of $I_{\sigma}$.\\

Using the fact that there are nowhere dense sets of positive measure, we obtain a collection
$\{ A(\sigma, i): (\sigma, i) \in \B\}$ such that the following conditions hold:
\begin{itemize}
\item each $ A(\sigma, i)$ is a closed subset of $I_{\sigma}$,
\item $\lambda( A(\sigma,i) )>0$, and
\item  if $(\sigma,i) \neq (\tau, j)$, then $A(\sigma,i) \cap A(\tau,j) = \emptyset$.\\
\end{itemize}

We also enumerate the standard orthonormal basis of $\ell_2$  as $\{e(\sigma,i):
(\sigma, i) \in \B\}$ so the following conditions hold:
\begin{itemize}
\item each of $e(\sigma, i)$ is of the form $(0,\ldots, 1, 0, \ldots)$, and
\item $e(\sigma, i) \perp e(\tau,j)$ if $(\sigma,i) \neq (\tau, j)$.\\
\end{itemize}

Let $c(\sigma,i) \in \reals$ for $(\sigma, i) \in \B$. Then, using these coefficients we define
a special type of function as below:
\[f= \sum_{(\sigma,i) \in \B} c(\sigma,i) \cdot \frac{1}{\lambda(A(\sigma, i))} \cdot e(\sigma, i)\cdot\chi_{A(\sigma, i)}.\]
Functions of these type will be called {\it basic functions.}

If $\tau \in \finseq$ and $f$ is as above, we define
\[f_{|\tau} = \sum_{(\sigma,i) \in \B, \ \sigma \in [\tau] } c(\sigma,i) \cdot \frac{1}{\lambda(A(\sigma, i))} \cdot e(\sigma, i)\cdot\chi_{A(\sigma, i)}.\]

We will  freely use the following facts about basic functions.
\begin{lem}\label{basic} Let
\[f= \sum_{(\sigma,i) \in \B} c(\sigma,i) \cdot \frac{1}{\lambda(A(\sigma, i))} \cdot e(\sigma, i)\cdot\chi_{A(\sigma, i)}.\]
Then,
\begin{enumerate}
\item $f \in \pettis$ iff  $\sum_{(\sigma,i) \in \B} c(\sigma , i) ^2  < \infty $.
Moreover, in the case that $f \in \pettis$, we have that
\[ \left \| \int_{[0,1]} f \right \| = \sqrt{\sum_{(\sigma,i) \in \B} c(\sigma , i) ^2 }
\]
\item If $I, J$ are subintervals of $[0,1]$ with $I \subseteq J$, then
\[  \left \| \int _{I} f \right \| \le \left \| \int _{J} f \right \| .
\]
\item If $\tau \in \finseq$, then
\[ \left \| \int _{I_{\tau}} f \right \|  \ge \left \|  \int _{I_{\tau}} f_{| \tau}  \right \|  =  \sqrt{ \sum_{(\sigma,i) \in \B, \ \sigma \in [\tau] } c(\sigma , i) ^2 }.\]
\end{enumerate}

\end{lem}
\begin{proof} $(1)$ follows from  Lemma \ref{basic1} and from the fact that the standard base in $l_2$ is unconditional (see \cite{singer}). $(2)$ and $(3)$ follow from the fact that
\[   \int _{I} f= \sum_{(\sigma,i) \in \B} c(\sigma,i) \cdot e(\sigma,i) \cdot \frac{\lambda(A(\sigma, i)\cap I)}{\lambda(A(\sigma, i))} \]

\end{proof}

We let
\[ \D = \{ \n| \n: \nat \rightarrow \nat, \n (i) \le i\}.
\]

For each $\n \in \D$, we define a special type of basic function. Namely, let
 $\fn: [0,1] \rightarrow \ell_2$ be defined as
\[ \fn = \sum_{k=0}^{\infty} \ \sum_{\sigma \in \2^{k}} \frac{1}{(k+1)2^{k/2}} \cdot \frac{1}{\lambda(A(\sigma, \n(k)))} \cdot e(\sigma, \n(k))\cdot\chi_{A(\sigma, \n(k))}. \]

\begin{lem}\label{integral exact} For each $\n \in \D$, we have that $\fn$ is Pettis integrable and for $\tau \in \2 ^i$, we have that
\[ \left  \| \int_{I_{\tau}}f(\n)_{| \tau}  \right  \|  = 2^{-i/2} \sqrt{\sum_{k=i}^{\infty}  \left [ \frac{1}{(k+1)^2}  \right ]}.
\]

\end{lem}

\begin{lem}\label{infinite sum} Let $\{\lambda_i\}$ be a sequence in $\ell_1$ and $\{f_i\}$ be such
that $f_i = f(\n_i)$ for some $\n_i \in \D$. Then,
\[ f= \sum_{i=1}^{\infty} \lambda_i f_i \in \pettis \ \ \ \  \  \mbox{  and   } \ \ \ \ \ \forall \tau \in \finseq, \ \ \   \left  \|  \int_{I _{\tau}} f_{| \tau}
 \right \| \le  \sum_{i=1}^{\infty}  | \lambda_i | \left \| \int_{I _{\tau}} f_{i| \tau}\right \|.\]
\end{lem}
\begin{proof} We will show that  $f $ is Pettis integrable and
\[ \left \| \int _{[0,1]}  f
 \right \| \le  \sum_{i=1}^{\infty}  | \lambda_i | \left \| \int_{[0,1]} f_i\right \|,
\]
as restricting the function to $I_{\tau}$ does not alter the basic computations.  We first note that
\[ f = \sum_{k=0}^{\infty} \sum_{\sigma \in \2^{k}}  \sum_{j=0}^{k} \frac{1}{(k+1)2^{k/2}} \cdot \frac{d(k, j)}{\lambda(A(\sigma, j))} \cdot e(\sigma, j)\cdot\chi_{A(\sigma, j)},
\]
where $d(k, j) = \sum_{\{i: \n_i(k) =j\}}\lambda_i$.
Hence, $ | d(k,j)| \le \sum_{\{i: \n_i(k)  =j\}} | \lambda_i |$.
Since $j \neq j'$ implies that $\{i: \n_i(k)= j\} \cap \{i: \n_i(k) =j'\} = \emptyset$, we have that
\[\sum_{j=0}^{k} |d(k,j)| ^2\le \sum_{j=0}^{k} \left (  \sum_{\{i: \n_i(k) =j\}} | \lambda_i | \right ) ^2 \le  \left ( \sum_{i=1}^{\infty} | \lambda_i  | \right ) ^2. \]
We note that $f$ is Pettis integrable provided that the series
\[\sum_{k=0}^{\infty} \sum_{\sigma \in \2^{k}}  \sum_{j=0}^{k} \frac{| d(k, j)| ^2 }{(k+1)^22^k} < \infty.
\]
Indeed,
\begin{eqnarray*}
\sum_{k=0}^{\infty} \sum_{\sigma \in \2^{k}}  \sum_{j=0}^{k} \frac{| d(k, j)| ^2 }{(k+1)^22^k}
& = & \sum_{k=0}^{\infty} \sum_{\sigma \in \2^{k}}  \frac{1}{(k+1)^22^k} \sum_{j=0}^{k} | d(k, j)| ^2 \\
& \le & \sum_{k=0}^{\infty} \sum_{\sigma \in \2^{k}}  \frac{1}{(k+1)^22^k}  \left ( \sum_{i=1}^{\infty} | \lambda_i | \right ) ^2 \\
& = &  \left ( \sum_{i=1}^{\infty} | \lambda_i | \right ) ^2 \sum_{k=0}^{\infty}  \frac{1}{(k+1)^2}   \\
& < & \infty.
\end{eqnarray*}
Hence we have that $f$ is Pettis integrable. Moreover, as  by Lemma \ref{integral exact},  $\|\int_{[0,1]}f_i\| = \sqrt{\sum_{k=0}^{\infty}  \frac{1}{(k+1)^2}}$ for all $i=1,2,...$, we have that
\begin{eqnarray*}
\left \|  \int_{[0,1]} f \right \| & =  & \left ( \sum_{k=0}^{\infty} \sum_{\sigma \in \2^{k}}  \sum_{j=0}^{k} \frac{| d(\sigma, j)| ^2 }{(k+1)^22^k} \right ) ^{1/2}\\
& \le &  \left ( \left ( \sum_{j=1}^{\infty} | \lambda_j | \right ) ^2 \sum_{k=0}^{\infty}  \frac{1}{(k+1)^2} \right) ^{1/2}\\
& = & \sum_{j=1}^{\infty} | \lambda_j |  \sqrt{ \sum_{k=0}^{\infty}  \frac{1}{(k+1)^2}} \\
& = & \sum_{j=1}^{\infty} | \lambda_j |  \left \|\int_{[0,1]} f_j  \right \|,
\end{eqnarray*}
 completing the proof.
\end{proof}
\begin{proof}({\bf of Theorem~\ref{mainthm} for $X =\ell_2 $.})
We first obtain a subfamily $\{ \n_t \}$ of $\D$, $0 < t < 1$, such that the following
condition holds:
\begin{eqnarray}
 s, t \in (0,1) \ \& \ s \neq t \implies  \{k: \n_s (k) = \n_{t} (k) \} \mbox { is finite. }  \label{eq: ad}
 \end{eqnarray}
This may be done in the following fashion. For each $0 < t < 1$, consider the line $L_t$ going through the
origin with slope $t$.  For each $k\in \nat$ choose $\n _t(k) \in  \nat \cap [0,k] $ so that $L_t(k) -1<  \n_t(k)<L_t(k)+1$.
Then, the collection $\{ \n_t\}$ has the desired property.\\

Let ${\mathcal V} = \{f_{\n_t}: t \in (0,1) \}$. We will now show  that ${\mathcal V}$ has the desired properties.\\

It is clear that ${\mathcal V}$ has the cardinality that of the continuum. By $(\ref{eq: ad})$ it follows that no non-trivial finite linear combination
of elements of  $\{ f_{\n_t} \}$ is the zero function. Hence, ${\mathcal V}$ is linearly independent. That ${\mathcal V}$
satisfies conclusion $(3)$ of the theorem follows from Lemma~\ref{infinite sum}.

Finally, to verify conclusion (4) of the theorem, let us show that if $f$ is not the zero function.
then  for all $x \in [0,1]$, we have
\[ \lim_{h\rightarrow 0} \left \| \frac{F(x+h)-F(x)}{h} \right \|= \infty, \]
where $F$ is the primitive of $f$.
As $f$ is not the zero function, we may assume, by renumeration,  that all of the $\{f_i\}$'s are distinct
and $\lambda_1 \neq 0$.   Moreover, we lose no generality by assuming
that $\lambda_1 =1$. Let $i_0$ be such that $\sum_{i=i_0}^{\infty} | \lambda_i| <  \frac{1}{2}$. Let $t_i$ be such that $f_i = f_{\n_{t_i}}$. Let $M >0$.
By our choice of $\{\n_{t_i}\}$, we have that
there exists a positive integer $l$ such that for all $k \ge l$, $\sigma \in \2 ^k$ and $1\le i < i' \le i_0$, we have
that $\n_i(k) \neq \n_{i'} (k)$ and hence $e(\sigma,n_i(k)) \perp e(\sigma, n_{i'}(k))$.  Choose $k_0$ so large so that $k_0 > l$ and $2^{i/2-3} \sqrt{ \sum_{k=i}^{\infty}  \frac{1}{(k+1)^2}} > M$  for all $i > k_0$. Let $\delta =2^{-k_0}$. Let $ h \in \reals $ be such
that $0 < |h| < \delta$. We wish to show that \[ \left \| \frac{ F(x+h)-F(x) }{h}  \right \| > M. \]
Without loss of generality, we may assume that $h>0$. Let $j$ be the smallest positive integer so that there is $\tau \in \2^j$ so that $I_{\tau} \subseteq [x, x+h]$. We note that $j> k_0$ and $h<4 \cdot 2^{-j}$.

\begin{eqnarray}
 \left  \| F(x+h) -  F(x)   \right \|  & = & \left \| \int _{[x,x+h]} f \right \|  \nonumber \\
\label{eq: basic1} & \ge & \left \| \int _{I_{\tau}} f \right \| \\
\label{eq: basic2} & \ge & \left \| \int _{I_{\tau}} f _{|\tau}\right \| \\
& = &  \left \| \int _{I_{\tau}} \sum_{i=1}^{i_0} \lambda_i f_{i|\tau} + \int _{I_{\tau}} \sum_{i=i_0}^{\infty} \lambda_i f_{i|\tau} \right \| \nonumber \\
\label{eq: triangle1} & \ge & \left \| \int _{I_{\tau}} \sum_{i=1}^{i_0} \lambda_i f_{i|\tau} \right \| - \left \|  \int _{I_{\tau}}\sum_{i=i_0}^{\infty} \lambda_i f_{i|\tau} \right \|\\
\label{eq: infinite sum4}& \ge & \left \| \int _{I_{\tau}} \sum_{i=1}^{i_0} \lambda_i f_{i|\tau} \right \| - \sum_{i=i_0}^{\infty} | \lambda_i | \left \|  \int _{I_{\tau}} f_{i|\tau} \right \|\\
\label{eq: almost disjoint} & \ge & \left \| \int _{I_{\tau}} f_{1|\tau} \right \|  - \sum_{i=i_0}^{\infty} | \lambda_i | \left \|  \int _{I_{\tau}} f_{i|\tau} \right \|\\
\label{eq: exact}& = & \left ( 1 - \sum_{i=i_0}^{\infty} | \lambda_i | \right ) 2^{-j/2} \sqrt{\sum_{k=j}^{\infty}  \left [ \frac{1}{(k+1)^2}  \right ]}
\end{eqnarray}
Let us give some explanations for the above inequalities. Let us first observe that $f$ is a basic function.  Inequalities  (\ref{eq: basic1}) and (\ref{eq: basic2}) above follow from
Lemma \ref{basic}. Inequality (\ref{eq: triangle1}) is simply the triangle inequality. Meanwhile, inequality (\ref{eq: infinite sum4}) follows from Lemma~\ref{infinite sum}. Inequality
(\ref{eq: almost disjoint}) holds because $j > k_0 >l$ which implies that for all $1 \le i < i' \le i_0$ we have that $e(\sigma,n_i(|\sigma|)) \perp e(\sigma, n_{i'}(|\sigma|))$ for all $\sigma$
extensions of $\tau$ and $(\sigma,k) \in \B$. Finally, estimate (\ref{eq: exact}) follows from Lemma~\ref{integral exact}
Now using the above estimate and the fact that $h<4 \cdot 2^{-j}$, we obtain that
\begin{eqnarray*}
 \left \| \frac{ F (x+h)-F (x) }{h}  \right \| & > & \frac {\left ( 1 - \sum_{i=i_0}^{\infty} | \lambda_i | \right ) 2^{-j/2} \sqrt{\sum_{k=j}^{\infty}  \left [ \frac{1}{(k+1)^2}  \right ]}}{4 \cdot 2^{-j}}\\
 & \ge &  2^{j/2-3} \sqrt{ \sum_{k=j}^{\infty}  \frac{1}{(k+1)^2}}\\
 & > & M.
\end{eqnarray*}
\end{proof}
\vspace{2ex}

\section{The general case}
The proof of the general case is similar to the $\ell_2$ case, the main difference being the lack of orthonormal basis in an arbitrary Banach space $X$. We use the notation and the terminology of the previous section with one exception:
$\{e(\sigma, i): (\sigma, i) \in \B\}$ will have to be constructed with the help of Dvoretzky's theorem. The calculations are also more involved.

\begin{lem}\label{basic sequence}
Let $\{b_n\}$ be a basic sequence in $X$. Then there is $K>1$ such that for all
$\lambda_0, \lambda_1 \ldots \in \reals$ for which $\sum_{i=0}^{\infty} \lambda_i b_i$ converges in $X$ we have that
\[ \left \| \sum_{i=0}^{\infty} \lambda_i b_i \right \| \ge  \frac{1}{K} \left \|  \sum_{i=k}^{l} \lambda_i b_i \right \| \mbox{ for all } k < l \in \nat  \cup \{  \infty \}.
\]
\end{lem}
\begin{proof} Let $Y$ be the subspace of $X$ generated $\{b_n\}$. Let $\pi_k: Y \rightarrow Y$
be the natural projection defined as
\[ \pi_k \left (\sum_{i=0}^{\infty} \lambda_i b_i \right ) = \sum_{i=0}^{k} \lambda_i b_i.
\]
($\pi_{\infty}$ is simply the identity map.)
It is well-known that $\pi_k : Y \rightarrow Y$ is a bounded linear map. Moreover, it is known that $\{\pi_k: k \in \nat \cup \{ \infty\} \} $ is uniformly bounded.  For each $k< l \in \nat \cup \{ \infty\}$, let
\[ q_{k,l} \left (\sum_{i=0}^{\infty} \lambda_i b_i \right ) = \sum_{i=k}^{l} \lambda_i b_i.
\]
Then, $q_{k,l} = \pi_l - \pi_{k-1}$. Hence, the family $\{q_{k,l}\}$ is uniformly bounded and the lemma follows.
\end{proof}

 Throughout this section $\{b_n\}$ and $K$ are as above in the Lemma~\ref{basic sequence}.\\

We now proceed to construct $\{e(\sigma, j): (\sigma, i) \in \B\}$.  Let $\{n_k\}_{k=0}^{\infty}$ be a strictly increasing
sequence such that $n_0=0$.  Let
 \[ \B _k=\{(\sigma, i): \sigma \in \2^{[n_k,n_{k+1})}, \ \ \ 0 \le i \le |\sigma|\}.
 \]
 We call $\B_k$ the $k^{th}$ block.

 The following lemma follows from Dvoretzky's  theorem (see \cite{dvoretzky}).
 \begin{lem}\label{choosingen} Let $\{n_k\}_{k=0}^{\infty}$, $\{m_k\}_{k=0}^{\infty}$ and $B_k$ be as above.
 There exits  $\{e(\sigma,i): (\sigma, i) \in B_k\}$ such that the following
 conditions hold:

 \[  \{e(\sigma,i): (\sigma, i) \in \B_k\} \subseteq span\{b_j: j \in [m_k, m_{k+1})\},
\]
\[    \forall \lambda(\sigma, i) \in \reals \ \ \ \ \  \sqrt {\sum _{(\sigma, i) \in \B_k} \lambda (\sigma,i)^2} \ge  \left \| \sum _{(\sigma, i) \in \B_k} \lambda (\sigma,i) e(\sigma, i) \right \| \ge \frac{1}{2}
 \sqrt {\sum _{(\sigma, i) \in \B_k} \lambda (\sigma,i)^2}.
\]
 \end{lem}

 As before, we let
 \[f= \sum_{(\sigma,i) \in \B} c(\sigma,i) \cdot \frac{1}{\lambda(A(\sigma, i))} \cdot e(\sigma, i)\cdot\chi_{A(\sigma, i)},\]
 and for $\tau \in \2 ^{n_k}$, let
 \[f_{|\tau} = \sum_{(\sigma,i) \in \B_, \ \sigma \in [\tau] } c(\sigma,j) \cdot \frac{1}{\lambda(A(\sigma, j))} \cdot e(\sigma, j)\cdot\chi_{A(\sigma, j)}.\]

\begin{lem}\label{general estimates} Let $f$ and $f_{|\tau}$ be as above with $\tau \in \2 ^{n_k}$.
\begin{enumerate}
\item If  $\sum_{j=0}^{\infty}\sqrt{\sum_{(\sigma,i) \in \B_j} c(\sigma , i) ^2 } < \infty $,
then $f\in \pettis$.
\item
\[ \frac{1}{2K}  \sqrt{ \sum_{(\sigma,i) \in \B_k, \ \sigma \in [\tau] } c(\sigma , i) ^2 }  \le \left \| \int _{I_{\tau}} f_{|\tau} \right\|  \le \sum_{j=k}^{\infty} \sqrt{ \sum_{(\sigma,i) \in \B_j , \ \sigma \in [\tau] } c(\sigma , i) ^2 }   \]
\item
\[\left \| \int _{I_{\tau}} f \right \|  \ge  \frac{1}{K} \left \| \int _{I_{\tau}} f_{|\tau} \right \|.\]

\end{enumerate}

\end{lem}
\begin{proof} To see (1)
recall that $f \in \pettis$ provided that $\sum_{(\sigma,i) \in \B} c(\sigma,i) \ e(\sigma, i) $
converges unconditionally.  To see this, let $\epsilon(\sigma, i) \in \{-1,1 \}$ for all $(\sigma,i) \in \B$.
Then,

 \begin{eqnarray}
\label{eq: gtriangle 1} \left \| \sum_{(\sigma,i) \in \B} \epsilon(\sigma, i) c(\sigma,i) \ e(\sigma, i) \right \| & \le &
\sum_{j=0}^{\infty}  \left \| \sum_{(\sigma,i) \in \B_j} \epsilon(\sigma, i) c(\sigma,i) \ e(\sigma, i) \right \| \\
\label{eq: dvoretzky 1} & \le & \sum_{j=0}^{\infty}   \sqrt {\sum _{(\sigma, i) \in \B_j}   c(\sigma,i) ^2} \\
& < & \infty.
 \end{eqnarray}
Note that inequality (\ref{eq: gtriangle 1}) is simply the triangle inequality and inequality (\ref{eq: dvoretzky 1}) follows from Lemma~\ref{choosingen}.

Let us now verify (2). To obtain the lower bound,  we observe that
\begin{eqnarray} \left \|  \int _{I_{\tau}} f_{|\tau} \right \|  & = &  \left \| \sum_{(\sigma,i) \in \B,\  \sigma \in [\tau]} c(\sigma,i) \cdot \frac{\lambda(A(\sigma, i)\cap I_{\tau})}{\lambda(A(\sigma, i))} \cdot e(\sigma, i) \right \|\\
& =&  \left \| \sum_{(\sigma,i) \in \B,\  \sigma \in [\tau]} c(\sigma,i)  \cdot e(\sigma, i) \right \| \\
& =& \left \| \sum_{j=k}^{\infty}\  \sum_{(\sigma,i) \in \B_j,\  \sigma \in [\tau]} c(\sigma,i)  \cdot e(\sigma, i) \right \|\\
\label{eq: basic 0} & \ge & \frac{1}{K} \left \|  \sum_{(\sigma,i) \in \B_k,\  \sigma \in [\tau]} c(\sigma,i)  \cdot e(\sigma, i) \right \|\\
\label{eq: dvoretzky 1.1} & \ge & \frac{1}{2K}  \sqrt{ \sum_{(\sigma,i) \in \B_k, \ \sigma \in [\tau] } c(\sigma , i) ^2 }
\end{eqnarray}
Inequality (\ref{eq: basic 0}) follows from the fashion in which $e(\sigma,i)$'s were chosen and  Lemma~\ref{basic sequence}. Meanwhile Inequality (\ref{eq: dvoretzky 1.1}) follows from Lemma~\ref{choosingen}.
To obtain the upper bound, we observe that
\begin{eqnarray} \left \|  \int _{I_{\tau}} f_{|\tau} \right \|  & = & \left \| \sum_{j=k}^{\infty}\  \sum_{(\sigma,i) \in \B_j,\  \sigma \in [\tau]} c(\sigma,i)  \cdot e(\sigma, i) \right \|\\
\label{eq: gtriangle+} & \le  & \sum_{j=k}^{\infty} \left \|  \sum_{(\sigma,i) \in \B_j,\  \sigma \in [\tau]} c(\sigma,i)  \cdot e(\sigma, i) \right \|\\
\label{eq: dvoretzky 1.2} & \le & \sum_{j=k}^{\infty}  \sqrt{ \sum_{(\sigma,i) \in \B_j, \ \sigma \in [\tau] } c(\sigma , i) ^2 }.
\end{eqnarray}
Inequality (\ref{eq: gtriangle+}) simply is the triangle inequality. Inequality (\ref{eq: dvoretzky 1.2}) follows from Lemma~\ref{choosingen}.

Finally, let us now prove (3).  We note that
\begin{eqnarray*} \left \|  \int _{I_{\tau}} f  \right \| & = &  \left \| \sum_{(\sigma,i) \in \B, \ \tau \in [\sigma] } c(\sigma,i) \cdot \frac{\lambda(A(\sigma, i)\cap I_{\tau})}{\lambda(A(\sigma, i))} \cdot e(\sigma, i) +  \sum_{(\sigma,i) \in \B, \  \sigma \in [\tau ]} c(\sigma,i)  \cdot e(\sigma, i)  \right \| \\
& = &  \left \| \sum_{j=0}^{k-1} \ \sum_{(\sigma,i) \in \B_j, \ \tau \in [\sigma] } c(\sigma,i) \cdot \frac{\lambda(A(\sigma, i)\cap I_{\tau})}{\lambda(A(\sigma, i))} \cdot e(\sigma, i) +   \sum_{j=k}^{\infty} \ \sum_{(\sigma,i) \in \B_j, \  \sigma \in [\tau ]} c(\sigma,i)  \cdot e(\sigma, i)  \right \|\\
& \ge & \frac{1}{K} \left \|  \sum_{j=k}^{\infty} \  \sum_{(\sigma,i) \in \B_j, \  \sigma \in [\tau ]} c(\sigma,i)  \cdot e(\sigma, i)  \right \| \\
& = &  \frac{1}{K}\left \|  \int _{I_{\tau}} f_{|\tau}  \right \|
\end{eqnarray*}
Again in the inequalities above we make a use of the fashion in which $e(\sigma,i)$'s where chosen and Lemma~\ref{basic sequence}.

\end{proof}
Recall that
\[ \D = \{ \n| \n: \nat \rightarrow \nat, \n (i) \le i\}.\\
\]
and
\[ \fn = \sum_{k=0}^{\infty} \ \sum_{\sigma \in \2^{k}} \frac{1}{(k+1)2^{k/2}} \cdot \frac{1}{\lambda(A(\sigma, \n(k)))} \cdot e(\sigma, \n(k))\cdot\chi_{A(\sigma, \n(k))}. \]
For the sake of notational convenience, for each $k \in \nat$, let
 \[ u_k \equiv  \sqrt{\sum_{i=n_k}^{n_{k+1}-1}   \frac{1}{ (i+1)^2 }   }.\]
Henceforth, we assume that $\{n_k\}$ is rapidly increasing so that
\[ u_{k+1} < \frac{1}{3}u_k.\]

\begin{lem}\label{particular estimates} For each $\n \in \D$, we have that $\fn$ is Pettis integrable and for $\tau \in \2 ^{n_k}$, we have that
\[  \frac{1}{2K} 2^{-\frac{n_k}{2}}u_k   \le \left  \| \int_{I_{\tau}}f(\n)_{|\tau} \right \| \le  \frac{3}{2} 2^{-\frac{n_k}{2}}u_k.
\]
\end{lem}
\begin{proof} Let us compute the lower estimate first. By Lemma~\ref{general estimates} we have that
\begin{eqnarray*}
\left \| \int_{I_{\tau}}f(\n)_{|\tau} \right \| & \ge & \frac{1}{2K}  \sqrt{ \sum_{i=n_k}^{n_{k+1}-1}2^{i-n_k}  \left [ \frac{1}{{(i+1)2^{i/2}}} \right ] ^2}\\
& = & \frac{1}{2K} 2^{-\frac{n_k}{2}} \sqrt{ \sum_{i=n_k}^{n_{k+1}-1}  \left [ \frac{1}{{(i+1)}} \right ] ^2}\\
& = &  \frac{1}{2K} 2^{-\frac{n_k}{2}} u_k.
\end{eqnarray*}
Again, using the upper estimate in Lemma~\ref{general estimates} we have that

\begin{eqnarray*}
\left \| \int_{I_{\tau}}f(\n)_{|\tau} \right \| & \le  & \ \sum_{j=k}^{\infty} \sqrt{ \sum_{i=n_k}^{n_{k+1}-1}2^{i-n_k}  \left [ \frac{1}{{(i+1)2^{i/2}}} \right ] ^2}\\
& \le &  2^{-n_k/2} \sum_{j=k}^{\infty} \sqrt{ \sum_{i=n_k}^{n_{k+1}-1}  \left [ \frac{1}{(i+1)^2} \right ] ^2}\\
& \le & 2^{-n_k/2} \sum_{j=k}^{\infty} u_j\\
& \le & 2^{-n_k/2} \sum_{j=k}^{\infty} u_k  \left (\frac{1}{3} \right )^{j-k}\\
& \le & \frac{3}{2} 2^{-n_k/2}u_k.
\end{eqnarray*}
\end{proof}
\begin{lem}\label{infinite sum general} Let $\{\lambda_i\}$ be a sequence in $\ell_1$ and $\{f_i\}$ be such
that $f_i = f(\n_i)$ for some $\n_i \in \D$. Then,
\[ f= \sum_{i=1}^{\infty} \lambda_i f_i \in \pettis \ \ \ \  \  \mbox{  and   } \ \ \ \ \ \forall \tau \in \finseq, \ \ \   \left  \|  \int_{I _{\tau}} f
 \right \| \le  \sum_{i=1}^{\infty}  | \lambda_i | \left \| \int_{I _{\tau}} f_{i}\right \|.\]
\end{lem}
\begin{proof} We only need to show that $f \in \pettis$ as the second part simply follows from the general theory of integration.
  We first note that
\[ f = \sum_{k=0}^{\infty} \sum_{\sigma \in \2 ^k} \sum_{j=0}^k  \frac{1}{(k+1)2^{k/2}} \cdot \frac{d(k, j)}{\lambda(A(\sigma, j))} \cdot e(\sigma, j)\cdot\chi_{A(\sigma, j)},
\]
where $d(k, j) = \sum_{\{i: \n_i(k) =j\}}\lambda_i$.
Hence, $ | d(k,j)| \le \sum_{\{i: \n_i(k)  =j\}} | \lambda_i |$.
Since $j \neq j'$ implies that $\{i: \n_i(k)= j\} \cap \{i: \n_i(k) =j'\} = \emptyset$, we have that

\begin{eqnarray}
\sum_{j=0}^{k} |d(k,j)| ^2\le \sum_{j=0}^{k} \left (  \sum_{\{i: \n_i(k) =j\}} | \lambda_i | \right ) ^2 \le  \left ( \sum_{j=0}^{\infty} | \lambda_j  | \right ) ^2. \nonumber
\end{eqnarray}
Let us rewrite $f$. For $(\sigma, j) \in \B$, let $c(\sigma, j) = \frac{1}{(|\sigma|+1)2^{|\sigma|/2}} \cdot d(|\sigma|, j)$.
Then
\[f= \sum_{(\sigma,j) \in \B} c(\sigma,j) \cdot \frac{1}{\lambda(A(\sigma, j))} \cdot e(\sigma, j)\cdot\chi_{A(\sigma, j)}.
\]
By (1) of Lemma~\ref{general estimates}, we have that $f$ is Pettis integrable provided that
$\sum_{k=0}^{\infty}\sqrt{\sum_{(\sigma,j) \in \B_k} c(\sigma , j) ^2 } < \infty$.
Indeed,
\begin{eqnarray}
\sum_{k=0}^{\infty}\sqrt{\sum_{(\sigma,j) \in \B_k} c(\sigma , j) ^2 }
& = & \sum_{k=0}^{\infty} \sqrt{\sum_{i=n_k}^{n_{k+1}-1} \sum_{\sigma \in \2^i}  \sum_{j=0}^i \frac{1}{2^i (i+1)^2 } d(i,j) ^2 } \nonumber \\
& \le  & \sum_{k=0}^{\infty} \sqrt{\sum_{i=n_k}^{n_{k+1}-1} \sum_{\sigma \in \2^i}  \frac{1}{2^i (i+1)^2 }  \left ( \sum_{j=1}^{\infty} | \lambda_j  | \right ) ^2 } \nonumber \\
& = & \sum_{k=0}^{\infty} \sqrt{\sum_{i=n_k}^{n_{k+1}-1}   \frac{1}{ (i+1)^2 }  \left ( \sum_{j=1}^{\infty} | \lambda_j  | \right ) ^2 }\nonumber\\
& = & \left (  \sum_{j=1}^{\infty} | \lambda_j  | \right ) \sum_{k=0}^{\infty}   \sqrt{\sum_{i=n_k}^{n_{k+1}-1}   \frac{1}{ (i+1)^2 }   } \nonumber \\
& =& \sum_{k=0}^{\infty} u_k \nonumber \\
 & < & \frac{3}{2} u_0 < \infty \nonumber .
\end{eqnarray}
\end{proof}

\begin{proof}({\bf of Theorem~\ref{mainthm} for $X$ arbitrary infinite dimensional Banach space.})

As before let  $\{ \n_t \}$, $0 < t < 1$, be a subfamily of $\D$ that satisfies the following condition:
\[s, t \in (0,1) \  \& \ s \neq t \implies  \{k: \n_s (k) = \n_{t} (k) \} \text { is finite. }\]
Let ${\mathcal V} = \{f_{\n_t}: t \in (0,1) \}$. We will now show  that ${\mathcal V}$ has the desired properties.

It is clear that ${\mathcal V}$ has the cardinality that of the continuum. As before we have that  ${\mathcal V}$ is linearly independent. That ${\mathcal V}$
satisfies conclusion $(3)$ of the theorem follows from Lemma~\ref{infinite sum general}.

Finally, to verify conclusion (4) of the theorem, let us show that if $f$ is not the zero function,
then  for all $x \in [0,1]$, we have
\[ \limsup_{h\rightarrow 0} \left \| \frac{F(x+h)-F(x)}{h} \right \|= \infty, \]
where $F$ is the primitive of $f$.
As $f$ is not the zero function, we may assume, by renumeration,  that all of the $\{f_i\}$'s are distinct
and $\lambda_1 \neq 0$.   Moreover, we lose no generality by assuming
that $\lambda_1 =1$. Let $i_0$ be such that $\sum_{i=i_0}^{\infty} | \lambda_i| <  \frac{1}{8K}$. Let $t_i$ be such that $f_i = f_{\n_{t_i}}$. Let $M >0$.
By our choice of $\{\n_{t_i}\}$, we have that
there exists a positive integer $l$ such that for all $k \ge l$, $\sigma \in \2 ^k$ and $1\le i < i' \le i_0$, we have
that $\n_{t_i}(\sigma) \neq \n_{t_{i'}} (\sigma)$.  Choose $k$ so large so that $k > l$ and $\frac{1}{4K^2}2^{\frac{n_k}{2}} u_k > M$. Let $\tau \in \2 ^{n_k}$ be such that $x \in I_{\tau}$. We wish to show that \[\frac{\left \|\int_{I_{\tau}}f \right \| }{|I_{\tau}|}> M. \]

By (3) of Lemma~\ref{general estimates}, we have that $ \left \|\int_{I_{\tau}}f \right \|  \ge \frac{1}{K}\left \|\int_{I_{\tau}}f_{|\tau} \right \| $. Hence, it suffices to show that
\[\frac{\left \|\int_{I_{\tau}}f_{|\tau} \right \| }{|I_{\tau}|}> K M.\]
We let
\[
\left \|\int_{I_{\tau}}f _{|\tau}\right \| = \left \| \int _{I_{\tau}} \left (  \sum_{i=1}^{\infty} \lambda_i f_{i}  \right ) _{|\tau} \right \| \ge  L_1 - L_2,\]
 where
\[ L_1 = \left \| \int _{I_{\tau}} \left (  \sum_{i=1}^{i_0}  \lambda_i f_{i}  \right ) _{|\tau}  \right \|, \ \ \ \
L_2= \left \| \int _{I_{\tau}} \left ( \sum_{i=i_0+1}^{\infty}  \lambda_i f_{i}  \right ) _{|\tau}\right \|.
\]
Let us write
\[\sum_{i=1}^{i_0} \lambda_i f_{i} = \sum_{(\sigma,j) \in \B}c(\sigma,j) \cdot d(\sigma ,j) \cdot \frac{1}{\lambda(A(\sigma, j))} \cdot e(\sigma, j)\cdot\chi_{A(\sigma, j)} \mbox { where }
\]

\[c(\sigma, j) = c(\sigma) = \frac{1}{(|\sigma|+1) 2 ^{\frac{|\sigma|}{2}}}, \mbox { and } d(\sigma, j) = \sum_{\{1 \le i \le i_0: \n_i(|\sigma|) =j\}}\lambda_i
\]

\[\mbox { if  } \{1 \le i \le i_0: \n_i(|\sigma|) =j\}\mbox { is not empty}, d(\sigma, j) =0 \mbox { otherwise. }\]
Then, we have that
\begin{eqnarray}
L_1 & =& \left \| \int _{I_{\tau}} \left (  \sum_{i=1}^{i_0}  \lambda_i f_{i}  \right ) _{|\tau}  \right \| \nonumber\\
& =& \left \| \int _{I_{\tau}} \left ( \sum_{(\sigma,j) \in \B} c(\sigma,j) \cdot d(\sigma,j) \cdot \frac{1}{\lambda(A(\sigma, j))} \cdot e(\sigma, j)\cdot\chi_{A(\sigma, j)} \right ) _{|\tau} \right \| \nonumber\\
\label{eq: dvoretzky 4}& \ge & \frac{1}{2K}\sqrt{ \sum_{(\sigma,j) \in \B_k, \ \sigma \in [\tau]} c(\sigma,j)^2 \cdot  d(\sigma,j)^ 2 }\\
\label{eq: smaller} & \ge & \frac{1}{2K}\sqrt{ \sum_{i=n_k}^{n_{k+1}-1} \sum_{\sigma \in \2 ^i, \ \sigma \in [\tau] }  c(\sigma)^2 }\\
& = & \frac{1}{2K}\sqrt{ \sum_{i=n_k}^{n_{k+1}-1} \sum_{\sigma \in \2 ^i, \  \sigma \in [\tau] }  \left ( \frac{1}{(|\sigma|+1) 2 ^{\frac{|\sigma|}{2}} } \right ) ^2 }\\
& = & \frac{1}{2K}  \sqrt{ \sum_{i=n_k}^{n_{k+1}-1}2^{i-n_k}  \left [ \frac{1}{{(i+1)2^{i/2}}} \right ] ^2} = \frac{2^{-\frac{n_k}{2}}}{2K} u_k.
\end{eqnarray}
Let us give some justifications for the inequalities. Inequality (\ref{eq: dvoretzky 4}) follows from (2) of Lemma~\ref{general estimates}. Let us now explain why inequality (\ref{eq: smaller}) holds.
We note that if $1 \le i < i' \le i_0$ and $\sigma \in [\tau]$, then $n_{t_i} (|\sigma|) \neq n_{t_{i'}}(|\sigma|)$.
Hence, $\sigma \in [\tau]$ implies that $d(\sigma, j) =0$ or $\lambda_i$ for some $1 \le i \le i_0$. Moreover, for
every $\sigma \in [\tau]$ there is a $j$ so that $d(\sigma, j) = \lambda_1=1$. Hence inequality (\ref{eq: smaller}) holds.
The rest of the equalities are basic computations.

We next obtain an upper estimate of $L_2$.
\begin{eqnarray}
L_2 & =&\left \| \int _{I_{\tau}} \left ( \sum_{i=i_0+1}^{\infty}  \lambda_i f_{i}  \right ) _{|\tau}\right \|\\
& =&  \left \|  \int _{I_{\tau}}\sum_{i=i_0+1}^{\infty} \lambda_i (f_{i})_{|\tau} \right \| \nonumber\\
\label{eq: infinite sum} &  \le & \sum_{i=i_0+1}^{\infty}| \lambda_i | \left \|  \int _{I_{\tau}} (f_{i})_{|\tau} \right \| \\
\label{eq: upper est} & \le & \left (  \sum_{i=i_0+1}^{\infty}| \lambda_i |  \right ) \left (  \frac{3}{2} 2^{-\frac{n_k}{2}}u_k \right )\\
& \le & \frac{1}{8K} \frac{3}{2}2^{-\frac{n_k}{2}}u_k <   \frac{1}{4K}2^{-\frac{n_k}{2}}u_k\nonumber
\end{eqnarray}
We note that inequality (\ref{eq: infinite sum}) follows from Lemma~\ref{infinite sum general} and inequality (\ref{eq: upper est}) follows from Lemma~\ref{particular estimates}.

Putting these estimates together we get that
\[ \left \|\int_{I_{\tau}}f_{|\tau} \right \| \ge  L_1 - L_2 > \frac{1}{4K} 2^{-\frac{n_k}{2}}u_k \]
This, in turn, implies that
\[ \frac{\left \|\int_{I_{\tau}}f_{|\tau} \right \|}{|I_{\tau}|}>  \frac{ \frac{1}{4K} 2^{-\frac{n_k}{2}}u_k}{2^{-n_k}}   \ge \frac{1}{4K} 2^{\frac{n_k}{2}}u_k  >KM , \]
completing the proof.
\end{proof}

\end{document}